\documentclass[11pt]{article}
\usepackage{amsmath,amsthm,amssymb}
\usepackage{xcolor} 
\usepackage{verbatim}
\usepackage{mathtools} 

\def\titlerunning#1{\gdef\titrun{#1}}
\makeatletter
\def\author#1{\gdef\autrun{\def\and{\unskip, }#1}\gdef\@author{#1}}
\def\address#1{{\def\and{\\\hspace*{18pt}}\renewcommand{\thefootnote}{}%
\footnote {#1}}%
\markboth{\autrun}{\titrun}}
\makeatother
\def\email#1{\hspace*{4pt}{\em e-mail}: #1}
\def\MSC#1{{\renewcommand{\thefootnote}{}%
\footnote{\emph{Mathematics Subject Classification (2010):} #1}}}
\def\keywords#1{\par\medskip
\noindent\textbf{Keywords:} #1}

\newtheorem{theorem}{Theorem}[section]
\newtheorem{prop}[theorem]{Proposition}
\newtheorem{cor}[theorem]{Corollary}
\newtheorem{lemma}[theorem]{Lemma}

\theoremstyle{definition}

\newtheorem{remark}[theorem]{Remark}

\numberwithin{equation}{section}

\def\cA{\mathcal A}

\def\cC{\mathcal C}

\def\cE{\mathcal E}
\def\cF{\mathcal F}

\def\cH{\mathcal H}
\def\cI{\mathcal I}

\def\cO{\mathcal O}

\def\cD{\mathcal D}

\def\cS{\mathcal S}

\def\PG{{\rm PG}}

\def\GF{{\rm GF}}

\def\PGL{{\rm PGL}}

\def\GL{{\rm GL}}

\def\Fq{{\mathbb{F}_q}}
\def\Fsqrtq{{\mathbb{F}_{\sqrt{q}}}}

\newcommand{\todo}[1]{{}}
\newcommand{\todol}[1]{{}}
\newcommand{\todof}[1]{{}}

\begin{document}


\baselineskip=16pt

\titlerunning{}

\title{On the independence number of graphs related to a polarity}

\author{Sam Mattheus
\and
Francesco Pavese 
\and
Leo Storme}

\date{}

\maketitle

\address{S. Mattheus: Department of Mathematics, Vrije Universiteit Brussel, Pleinlaan 2, 1050 Brussel, Belgium; \email{sam.mattheus@vub.ac.be} 
\and 
F. Pavese: Dipartimento di Meccanica, Matematica e Management, Politecnico di Bari, Via Orabona 4, 70125 Bari, Italy; \email{francesco.pavese@poliba.it}
\and
L. Storme: Department of Mathematics, Ghent University, Krijgslaan 281, 9000 Ghent, Belgium; \email{leo.storme@ugent.be}
}

\bigskip

\MSC{Primary 05C50; Secondary 05B25 05C69 05C35}


\begin{abstract}
We investigate the independence number of two graphs constructed from a polarity of $\PG(2,q)$. For the first graph under consideration, the Erd\H os--R\'enyi graph $ER_q$, we provide an improvement on the known lower bounds on its independence number. In the second part of the paper we consider the Erd\H os--R\'enyi hypergraph of triangles $\cH_q$. We determine the exact magnitude of the independence number of $\cH_q$, $q$ even. This solves a problem posed by Mubayi and Williford in \cite[Open Problem 3]{MW}. 

\end{abstract}

\section{Introduction}

An {\em independent set} (or a {\em coclique}) in a graph or hypergraph $\Gamma$ is a set of pairwise non-adjacent vertices. The {\em independence number} of $\Gamma$, denoted by $\alpha(\Gamma)$, is the size of the largest independent set in $\Gamma$. The aim of this paper is to improve the known lower bounds on the independence number of the Erd\H os--R\'enyi graph $ER_q$. We also consider a related hypergraph obtained from the Erd\H os--R\'enyi graph. 

Let $q$ be a prime power and let $\PG(n,q)$ be the $n$--dimensional projective space over the finite field $\Fq$. A {\em polarity} $\rho$ of the projective space $\PG(n,q)$ is an involutory bijective map sending points to hyperplanes and hyperplanes to points which reverses incidence. Two distinct points $P_1, P_2$ of $\PG(n,q)$ are said to be {\em conjugate} with respect to $\rho$ if $P_1 \in \rho(P_2)$. A point $P$ is called {\em absolute} with respect to $\rho$ if $P \in \rho(P)$. The {\em polarity graph} of $\PG(n,q)$ with respect to a polarity $\rho$ is the simple graph $(V,E)$ with vertex set equal to the set of points of $\PG(n,q)$ and such that for two distinct points $P_1, P_2$, we have that $\{ P_1, P_2 \} \in E$ if and only if $P_1 \in \rho(P_2)$. 
The projective space $\PG(n,q)$ is known to have the orthogonal polarity if $q$ is odd (which in turn, if $n$ is odd, can be either hyperbolic or elliptic), the pseudo polarity if $q$ is even, the symplectic polarity if $n$ is odd and the unitary polarity if $q$ is a square. Any other polarity of $\PG(n,q)$ is projectively equivalent to one of these, see \cite[Table 2.1]{H}.

The Erd\H os--R\'enyi graph $ER_q$ is either the orthogonal or the pseudo polarity graph of $\PG(2,q)$, according as $q$ is odd or even, respectively. It was introduced by Erd\H os and R\'enyi in \cite{ER} (and independently by Brown in \cite{B}) to solve a problem in extremal graph theory. Recall that if $G$ is a graph, then $\mathrm{ex}(n,G)$ denotes the largest number of edges a graph on $n$ vertices can have without containing $G$ as a subgraph. Any graph on $n$ vertices with $\mathrm{ex}(n,G)$ edges and which has no copy of $G$ as a subgraph is called {\em extremal}. Of particular interest is the behavior of $\mathrm{ex}(n,C_4)$ where $C_4$ denotes the cycle of length $4$. In \cite{ERS}, Erd\H os, R\'enyi and S\'os proved that $\mathrm{ex}(n,C_4) \sim \frac{1}{2} n^{3/2}$ using the graphs $ER_q$ for constructive lower bounds. F\"uredi later demonstrated in \cite{F1} and \cite{F2} that the graphs $ER_q$ are extremal when $q$ is even or $q > 13$. The graph $ER_q$ has also been used to solve a similar problem for hypergraphs (see \cite{LV}). These hypergraphs will be dealt with in the last section.

As pointed out by Mubayi and Williford in \cite{MW}, the question of determining the independence number of a polarity graph can be phrased as a simple question in finite geometry which seems interesting in itself: 

{\em Let $\rho$ be a polarity of $\PG(n,q)$, what is the maximum number of mutually non--conjugate points of $\PG(n,q)$ with respect to $\rho$?}

The next result summarizes lower and upper bounds on the size of the independence number of $ER_q$ studied in \cite{HW, MW}.
$$
\alpha(ER_q) \le
\begin{cases} 
q^{3/2}+\sqrt{q}+1 & \mbox{for all } q \\
q^{3/2}-q+\sqrt{q}+1 & \mbox{if } q \mbox{ is an even square}
\end{cases}
$$
$$
\alpha(ER_q) \ge
\begin{cases} 
\frac{q^{3/2}+q+2}{2} & \mbox{if } q \mbox{ is an odd square}\\
\frac{120 q^{3/2}}{73 \sqrt{73}} & \mbox{if } q \mbox{ is an odd non--square} \\
q^{3/2}-q+\sqrt{q} & \mbox{if } q \mbox{ is an even square} \\
\frac{q^{3/2}}{2 \sqrt{2}} & \mbox{if } q \mbox{ is an even non--square}
\end{cases}
$$
In this paper we provide an improvement on the lower bound: 
$$
\alpha(ER_q) \ge
\begin{cases} 
\frac{q^{3/2}-\sqrt{q}}{2} + q + 1 & \mbox{if } q \mbox{ is an odd square and }  \sqrt{q} \equiv -1 \pmod 4 \hspace{.5cm} \mbox{ (Corollary \ref{cor1})}\\
\frac{q^{3/2}+3 q}{2} + 1 & \mbox{if } q \mbox{ is an odd square and } \sqrt{q} \equiv 1 \pmod 4 \hspace{.8cm}\mbox{ (Corollary \ref{cor2})} \\
\frac{q^{3/2}}{\sqrt{2}} - q + \sqrt{\frac{q}{2}} & \mbox{if } q \mbox{ is an even non--square} \hspace{3.8cm}\mbox{(Corollary \ref{coreven})}
\end{cases}
$$

In the last section we consider the  Erd\H os--R\'enyi hypergraph of triangles $\cH_q$. This hypergraph $\cH_q$ is the $3$--graph whose vertex set is the set of non--absolute points of $V(ER_q)$ and whose edge set is the set of triangles in $ER_q$. It follows from the definition that $\alpha(\cH_q)$ is the order of the largest triangle--free induced subgraph of $ER_q$ which contains no absolute points. In \cite{P}, Parsons constructs a triangle--free induced subgraph of $ER_q$, $q$ odd, which contains no absolute points and has either $q(q+1)/2$ or $q(q-1)/2$ vertices according as $q \equiv -1\pmod 4$ or $q \equiv 1\pmod 4$, respectively. We will show the existence of a triangle--free induced subgraph of $ER_q$, $q$ even, which contains no absolute points and has $q(q+1)/2$ vertices. With this result we establish that the bound determined in \cite[Theorem 8]{MW} is essentially tight.

\section{On the independence number of $ER_q$}

In the following, the Desarguesian plane $\PG(2,q)$ is represented via homogeneous coordinates over the Galois field $\Fq$, i.e., represent the points of $\PG(2,q)$ by $\langle (x, y, z) \rangle$, $x,y,z \in \Fq$ and $(x, y, z) \neq (0, 0, 0)$, and similarly lines by $\langle [a, b, c] \rangle$, $a,b,c \in \Fq$ and $[a, b, c] \neq [0, 0, 0]$.
Incidence is given by $ax+by+cz=0$. To avoid awkward notation the angle brackets will be dropped. The point $U_i$ is the point with $1$ in the $i$-th position and $0$ elsewhere. A {\em quadric} of $\PG(2,q)$ is the locus of zeros of a quadratic polynomial, say $a_{11} X_1^2 + a_{22} X_2^2 + a_{33} X_3^2 + a_{12} X_1 X_2 + a_{13} X_1 X_3 + a_{23} X_2 X_3$. There exist four kinds of quadrics in $\PG(2,q)$, three of which are degenerate (splitting into lines, which could be in the plane $\PG(2,q^2)$) and one of which is non--degenerate \cite{H}. In what follows we will use the term {\em conic} to refer to a non--degenerate quadric of $\PG(2,q)$. In $\PG(2,q)$, a line meets a conic in either $0$, $1$ or $2$ points. A line is called either {\em external}, {\em tangent} or {\em secant}, according as it contains $0$, $1$ or $2$ points of the conic. If $q$ is even the tangent lines are concurrent at a point, called the {\em nucleus} of the conic. If $q$ is odd, the set of absolute points of an orthogonal polarity of $\PG(2,q)$ forms a conic. Vice versa, when $q$ is odd, a conic defines an orthogonal polarity. For further results on this topic, see \cite{H}.

To construct independent sets, we will start by considering a suitable subgroup $\cS$ of $\PGL(3,q)$ leaving the polarity invariant, i.e., if $h \in \cS$ and $P$ is a point of $\PG(2,q)$, then ${(P^h)}^\perp = {(P^\perp)}^h$, where $\perp$ is the polarity under consideration. Then we will consider a certain point--orbit of this subgroup $\cS$, say $\cO$, and we will prove that $\cO$ is a coclique. This method has the advantage that we only have to check that $|P^\perp \cap \cO|=0$ for a single point $P$ in the orbit $\cO$. For if $Q$ were a point in $\cO$ such that $Q^\perp \cap \cO$ contains a point $R$, then we would have $Q = P^h$ for a certain $h \in \cS$, from which follows that $R^{h^{-1}} \in P^\perp \cap \cO$. Therefore, $|P^\perp \cap \cO| = |Q^\perp \cap \cO|$.

\subsection{$q$ odd square}

Let $q$ be an even power of an odd prime. Let $\cC$ be the conic of $\PG(2,q)$ having the following equation:
$$
X_2^2 - X_1 X_3 = 0,
$$
and let $\perp$ denote the orthogonal polarity of $\PG(2,q)$ defined by $\cC$. The conic $\cC$ consists of the following set of points: $\{ (1, t, t^2) \,\, | \,\, t \in \Fq \} \cup \{ U_3 \}$. For a point $P = (x_1, x_2, x_3)$, its {\em polar line} is the line having equation $P^\perp: x_3 X_1-2 x_2 X_2 + x_1 X_3 = 0$. The points of the plane that are not on $\cC$ are either {\em external}, which means that they lie on two tangents of $\cC$, or {\em internal}, lying on no tangent of $\cC$. The polar line of a point $P$ is external, secant or tangent, according as $P$ is internal, external or on $\cC$, respectively. We will denote by $\cE$ the set of $q(q+1)/2$ external points of $\cC$ and by $\cI$ the set of $q(q-1)/2$ internal points of $\cC$.
Let $H$ be the stabilizer of $\cC$ in $\PGL(3,q)$. We shall find it helpful to work with the elements of $\PGL(3,q)$ as matrices in $\GL(3,q)$ and the points of $\PG(2,q)$ as column vectors, with matrices acting on the left. We recall the following well-known results, for which \cite{H} is the standard reference. 

\begin{lemma}\cite[Corollary 7.14]{H}
Let $H$ be the stabilizer group of the conic $\cC:X_2^2-X_1X_3=0$, then $H \cong \PGL(2,q)$ and the isomorphism is given by
$$
\left(
\begin{array}{cc}
a & b \\
c & d \\
\end{array}
\right)
\longleftrightarrow
\left(
\begin{array}{cccc}
a^2 & 2ac & c^2 \\
ab & ad+bc & cd \\
b^2 & 2bd & d^2 \\
\end{array}
\right),
$$
with $a, b, c, d \in \Fq$, $ad - bc \neq 0$.
\end{lemma}

\begin{remark}\label{tra}
The stabilizer in $\PGL(3,q)$ of a conic of $\PG(2,q)$, $q$ odd, has three orbits on the points of $\PG(2,q)$: the points of the conic, the external points and the internal points. Hence it acts transitively on the tangent lines, the secant lines and the external lines to the conic, respectively.  
\end{remark}

In the case when $q$ is an even power of an odd prime, the current best lower bound for the independence number of $ER_q$ is due to Mubayi and Williford in \cite{MW}:
$$
\alpha(ER_q) \ge \frac{q^{3/2}+q+2}{2}.
$$
Here we give a proof for an improved lower bound:
$$
\alpha(ER_q) \ge 
\begin{cases} 
\frac{q^{3/2}-\sqrt{q}}{2} + q + 1 & \mbox{if } q \mbox{ is an odd square and } \sqrt{q} \equiv -1 \pmod 4 \hspace{.6cm}\mbox{ (Corollary \ref{cor1})}\\
\frac{q^{3/2}+3 q}{2} + 1 & \mbox{if } q \mbox{ is an odd square and } \sqrt{q} \equiv 1 \pmod 4 . \hspace{.8cm}\mbox{ (Corollary \ref{cor2})}
\end{cases}
$$
We will achieve this result by showing the existence of a set of points consisting of the $q+1$ points of $\cC$ and $(q^{3/2}-\sqrt{q})/2$ or $(q^{3/2}+q)/2$ internal points forming a coclique in $ER_q$, according as $\sqrt{q} \equiv -1 \pmod 4$ or $\sqrt{q} \equiv 1 \pmod 4$, respectively. Since the set of absolute points is independent in $ER_q$ and an absolute point can never be adjacent to an internal point, we only need to find a set of internal points such that for every point in the set, its polar line is incident with none of the other internal points of the set. We will find such a set as an orbit of some subgroup of $H$.

Before proving this bound, we first state some useful results. The following result gives an easy criterion to determine whether a point is external or internal to the conic $\cC: X_2^2-X_1X_3=0$. Denote by $\Box_{q}$ the subset of $\Fq$, $q$ odd, consisting of its square elements and define $\Box_{\sqrt{q}}$ in a similar way.

\begin{lemma}\cite[Theorem 8.3.3]{H}
A point $P = (x_1, x_2, x_3) \in \PG(2,q) \setminus \cC$ is external to $\cC$ if and only if $x_2^2-x_1x_3 \in \Box_{q} \setminus \{ 0 \}$.
\end{lemma}

\begin{lemma}\label{square} 
Let $a \in \Fq$, then $a^{\sqrt{q}+1} \in \Box_{\sqrt{q}}$ if and only if $a \in \Box_{q}$.
\end{lemma}

\subsubsection{$\sqrt{q} \equiv -1 \pmod 4$}

Denote by $B \coloneqq \PG(2,\sqrt{q})$ the standard Baer subplane of $\PG(2,q)$ and let $c = \cC \cap B$ be the restriction of the conic $\cC$ to the Baer subplane $B$.

\begin{remark} 
Since every element of $\Fsqrtq$ is a square in $\Fq$, we have that every point in $B \setminus c$ is external to $\cC$.
\end{remark}

Let $G \coloneqq Stab_H(c)$. Then, it is known that $G \cong \PGL(2,\sqrt{q})$ is maximal in $H$, see \cite{Di, M}. It is easily seen that the group $G$ consists of the matrices
\[
\left(
\begin{array}{cccc}
a^2 & 2ac & c^2 \\
ab & ad+bc & cd \\
b^2 & 2bd & d^2 \\
\end{array}
\right) ,
\]
where $a,b,c,d \in \Fsqrtq, ad-bc \neq 0$.

We need to determine the orbits of the group $G$ on the points of $\PG(2,q) \setminus B$.
\begin{prop}\label{orb}
The group $G$ has the following orbits on the points of $\PG(2,q) \setminus B$:
\begin{itemize}
\item one orbit of size $q-\sqrt{q}$, consisting of the points of $\cC \setminus c$,
\item one orbit of size $q^{3/2}-\sqrt{q}$, consisting of the points of $\cE \setminus B$ on the tangent lines to $c$,
\item $\sqrt{q}-2$ orbits of size $(q^{3/2}-\sqrt{q})/2$, consisting of points of $\cE$,
\item $\sqrt{q}$ orbits of size $(q^{3/2}-\sqrt{q})/2$, consisting of points of $\cI$.
\end{itemize}
\end{prop}
\begin{proof}
Each point of $\PG(2,q) \setminus B$ lies on exactly one line having $\sqrt{q}+1$ points in common with $B$. If $P$ is a point of $\PG(2,q) \setminus B$ and $\ell_P$ is the line containing $P$ and having $\sqrt{q}+1$ points in common with $B$, then the stabilizer in $G$ of $P$ fixes $\ell_{P}$. Indeed, $\ell_P$ can be restricted to a line of $B$, so under $G$, $\ell_P$ has to map to a line with $\sqrt{q}+1$ points of $B$ again. As there is only one such line through $P$, it follows that it has to be fixed whenever $P$ is fixed. 

The point $U_2 \in B$ is external with respect to the Baer conic $c$. The stabilizer in $G$ of the line $U_2^\perp = \ell_1 = U_1 U_3$ is the dihedral group $D_1$ of order $2(\sqrt{q}-1)$ generated by

$$
\left(
\begin{array}{cccc}
1 & 0 & 0 \\
0 & d & 0 \\
0 & 0 & d^2 \\
\end{array}
\right) ,
\left(
\begin{array}{cccc}
0 & 0 & 1 \\
0 & 1 & 0 \\
1 & 0 & 0 \\
\end{array}
\right) ,
$$
with $d \in \Fsqrtq \setminus \{ 0 \}$.
An easy calculation shows that the stabilizer in $D_1$ of a point $P \in \ell_1$ is a group of order $4$, if $P \in B \setminus \cC$, or an involution if $P \notin B$. It follows that a $D_1$-orbit of a point of $(\ell_1 \cap B) \setminus \cC$ has size $(\sqrt{q}-1)/2$, whereas a $D_1$-orbit of a point of $\ell_1 \setminus B$ has size $\sqrt{q}-1$. On the other hand, $\ell_1 \setminus B$ contains $(q-1)/2$ points of $\cI$ and $(\sqrt{q}-1)^2/2$ points of $\cE$. Taking into account Remark \ref{tra}, we have that, under the action of the group $G$, there are $(\sqrt{q}+1)/2$ orbits of size $(q^{3/2}-\sqrt{q})/2$ consisting of points of $\cI$ and $(\sqrt{q}-1)/2$ orbits of size $(q^{3/2}-\sqrt{q})/2$ consisting of points of $\cE$ arising in this way, lying on a secant of $B$ to $c$.

Let $s$ be a non--square in $\Fsqrtq$. Let $Q$ be the point having coordinates $(-s,0,1)$. The point $Q \in B$ is internal with respect to the Baer conic $c$. The stabilizer in $G$ of the line $Q^\perp = \ell_2: X_1 = s X_3$ is the dihedral group $D_2$ of order $2(\sqrt{q}+1)$ generated by

$$
\left(
\begin{array}{cccc}
a^2 & 2sab & s^2b^2 \\
ab & a^2+sb^2 & sab \\
b^2 & 2ab & a^2 \\
\end{array}
\right) ,
\left(
\begin{array}{cccc}
1 & 0 & 0 \\
0 & -1 & 0 \\
0 & 0 & 1 \\
\end{array}
\right) ,
$$
with $a,b \in \Fsqrtq$ such that $a^2-sb^2 = \pm 1$.
An easy calculation shows that the stabilizer in $D_2$ of a point $P \in \ell_2$ is a group of order $4$, if $P \in B \setminus \cC$, or an involution if $P \notin B$. It follows that a $D_2$-orbit of a point of $(\ell_2 \cap B) \setminus \cC$ has size $(\sqrt{q}+1)/2$, whereas a $D_2$-orbit of a point of $\ell_2 \setminus B$ has size $\sqrt{q}+1$. On the other hand, $\ell_2 \setminus B$ contains $(q-1)/2$ points of $\cI$, $(\sqrt{q}+1)(\sqrt{q}-3)/2$ points of $\cE$ and $2$ points of $\cC \setminus c$. Taking into account Remark \ref{tra}, we have that, under the action of the group $G$, there are $(\sqrt{q}-1)/2$ orbits of size $(q^{3/2}-\sqrt{q})/2$ consisting of points of $\cI$ and $(\sqrt{q}-3)/2$ orbits of size $(q^{3/2}-\sqrt{q})/2$ consisting of points of $\cE$. Note that, since the two points of $\ell_2 \cap (\cC \setminus c)$ are interchanged by the group $D_2$, we have that, under the action of $G$, the points of the conic $\cC$ are partitioned into two orbits: the $\sqrt{q}+1$ points in $B$ and the remaining $q-\sqrt{q}$ points of $\cC \setminus c$.

In the same fashion, it is possible to prove that if $R$ is a point of $c$, then the stabilizer in $G$ of the line $R^\perp = \ell_3$ is a group $D_3$ of order $q-\sqrt{q}$. The line $\ell_3$ has $\sqrt{q}+1$ points in common with $B$ and it is tangent to $c$ at the point $R$. The stabilizer in $D_3$ of a point $P \in \ell_3$ is a group of order $\sqrt{q}-1$, if $P \in B \setminus \cC$, or is the identity if $P \notin B$. It follows that a $D_3$-orbit of a point of $(\ell_3 \cap B) \setminus \cC$ has size $\sqrt{q}$, whereas a $D_3$-orbit of a point of $\ell_3 \setminus B$ has size $q-\sqrt{q}$. On the other hand, $\ell_3 \setminus B$ contains $q-\sqrt{q}$ points of $\cE$. Taking into account Remark \ref{tra}, we have that, under the action of the group $G$, there is one orbit of size $q^{3/2}-\sqrt{q}$ consisting of points of $\cE \setminus B$ lying on a tangent line to $c$.
\end{proof}

Now we are ready to give a proof of the announced result. 

\begin{theorem}\label{goodorbit}
If $\sqrt{q} \equiv -1 \pmod 4$, then there are $(\sqrt{q}+1)/2$ $G$--orbits on internal points such that for every point $P$ in the orbit, $|P^G \cap P^\perp| = 0$.

\end{theorem}
\begin{proof}


Let $\ell$ be the line having equation $X_2 = 0$. Let $w \notin \Box_{q}$ and consider the internal point $P = (1, 0, w) \in \ell$. The polar line of $P$ has equation $P^\perp: w X_1 + X_3 = 0$.
\noindent
The line $\ell$ is secant to $c$ and from the proof of Proposition \ref{orb}, there are $(\sqrt{q}+1)/2$ $G$--orbits consisting of points of $\cI$, each having $\sqrt{q}-1$ points of $\ell$. We want to prove that each of these orbits has the required property. 

From Proposition \ref{orb}, the orbit of $P$ under the action of $G$ has size $(q^{3/2}-\sqrt{q})/2$. In particular $P^G = \{ (a^2+c^2 w, ab+cd w, b^2+d^2 w) \;\; | \;\; a, b, c, d \in \Fsqrtq, ad-bc \neq 0 \}$. Assume, by way of contradiction, that $|P^G \cap P^\perp| \neq 0$. Then there would exist $a,b,c,d \in \Fsqrtq$, with $ad-bc \neq 0$, such that
$$
c^2 w^2 + (a^2 + d^2) w + b^2 = 0.
$$
Now we have to distinguish whether $c$ is zero or not.
If $c \neq 0$, then $w^{\sqrt{q}}$ is also a root of the equation $c^2 X^2+ (a^2+d^2) X + b^2 = 0$ and $w^{\sqrt{q}+1} = b^2/ c^2 \in \Box_{\sqrt{q}}$. By Lemma \ref{square}, it follows that $w \in \Box_{q}$, a contradiction.
If $c = 0$, then, since $a^2+d^2 \neq 0$, we would have $w = -b^2 / (a^2+d^2) \in \Fsqrtq$, a contradiction.

\end{proof}

\begin{cor}\label{cor1}
If $q$ is an even power of an odd prime and $\sqrt{q} \equiv -1 \pmod 4$, then 
$$
\alpha(ER_q) \ge \frac{q^{3/2}-\sqrt{q}}{2}+q+1.
$$
\end{cor}
\begin{proof}
	We repeat the construction as discussed before: take the union of the $q+1$ absolute points and an orbit of internal points as described in Theorem \ref{goodorbit}. If $P$ is an absolute point, then $P^\perp$ contains only external points and $P$ itself. Therefore, an absolute point is never adjacent to an internal or another absolute point. On the other hand, an internal point of this set is not adjacent to any other internal point, as shown in Theorem \ref{goodorbit}. Therefore, the set under consideration is indeed an independent set in $ER_q$.
\end{proof}

\subsubsection{$\sqrt{q} \equiv 1 \pmod 4$}


Let $K$ be the subgroup of $H$ consisting of the matrices
$$
\left(
\begin{array}{cccc}
a^2 & 2ac & c^2 \\
0 & a & c \\
0 & 0 & 1 \\
\end{array}
\right) ,
$$
with $a,c \in \Fq$ and $a^{\sqrt{q}+1}=1$. Then $| K | = q (\sqrt{q}+1)$.
We need the following technical result.
\begin{lemma}\label{tec}
Let $\sqrt{q} \equiv 1 \pmod 4$. If $a^{\sqrt{q}+1} = 1$, then $a^2+1 \in \Box_{q}$.
\end{lemma}
\begin{proof}
Let $i \in \Fq$ be the root of a non--square $s$ in $\Fsqrtq$, hence $i^{\sqrt{q}} = -i$ and $i^2 = s$. Since $\Fq = \Fsqrtq[ i ]$, we have that $a = a_1+a_2 i$, for some $a_1, a_2 \in \Fsqrtq$. Hence, $a^{\sqrt{q}} = a_1 - a_2 i$ and $a^{\sqrt{q}+1} = a_1^2 - s a_2^2 = 1$. It follows that $a^2 + 1 = 2 a_1 (a_1 + a_2 i)$. Since $(a^2+1)^{\sqrt{q}+1} = 4 a_1^2 \in \Box_{\sqrt{q}}$, from Lemma \ref{square}, we have that $a^2+1 \in \Box_{q}$.
\end{proof}

\begin{theorem}
Let $\cO$ be a $K$--orbit on internal points, then for every point $P \in \cO$, we have that $|\cO \cap P^\perp| = 0$.
\end{theorem}
\begin{proof}
Let $\ell$ be the line having equation $X_2 = 0$. Let $w \notin \Box_{q}$ and consider the internal point $P = (1, 0, w) \in \ell$. An easy calculation shows that the stabilizer of $\ell$ in $K$ is a group $K_\ell$ of order $\sqrt{q}+1$ obtained from $K$ by putting $c=0$. In particular, under the action of $K_\ell$, the $(q-1)/2$ internal points of $\ell$ are permuted in $\sqrt{q}-1$ orbits of size $(\sqrt{q}+1)/2$ each. It follows that, under the action of $K$, the set of $q(q-1)/2$ internal points split into $\sqrt{q}-1$ orbits of size $q(\sqrt{q}+1)/2$ each. Without loss of generality, we only need to prove that $|P^K \cap P^\perp| = 0$, where $P^K = \{ (a^2+w c^2, c w, w) \,\, | \,\, a, c \in \Fq, a^{\sqrt{q}+1} = 1 \}$. The polar line of $P$ has equation $P^\perp: w X_1 + X_3 = 0$. Assume, by way of contradiction, that $|P^K \cap P^\perp| \neq 0$. Then there would exist $a,c \in \Fq$, with $a^{\sqrt{q}+1} = 1$, such that 
$$
c^2 w + a^2 + 1 = 0.
$$
If $c \neq 0$, then $w = -(a^2+1)/c^2$. By Lemma \ref{tec}, it follows that $-(a^2+1)/c^2 \in \Box_{q}$. Hence $w \in \Box_{q}$ a contradiction.
If $c = 0$, we would have 
$$
\left\{
\begin{array}{l}
a^{\sqrt{q}+1} = 1 \\
a^2 = -1 \\
\end{array}
\right. ,
$$
which is impossible, since $(\sqrt{q}+1)/2$ is odd.
\end{proof}
\begin{cor}\label{cor2}
If $q$ is an even power of an odd prime and $\sqrt{q} \equiv 1 \pmod 4$, then 
$$
\alpha(ER_q) \ge \frac{q^{3/2}+q}{2}+q+1.
$$
\end{cor}
\begin{proof}
	Similarly as in the proof of Corollary \ref{cor1}, we construct the independent set by taking the union of the $q+1$ absolute points and a $K$-orbit of internal points.
\end{proof}

\subsection{$q$ even non--square}\label{even}

If $q$ is even, the set of absolute points of a pseudo polarity of $\PG(2,q)$ forms a line, say $\ell$. Without loss of generality, we may assume that $\ell$ is the line having equation $X_1 = 0$. Let $\perp$ denote the pseudo polarity of $\PG(2,q)$ such that for a point $P = (x_1, x_2, x_3)$, its {\em polar line} is the line having equation $P^\perp: x_1 X_1 + x_3 X_2 + x_2 X_3 = 0$. Let $H$ be the subgroup of $\PGL(3,q)$ leaving the polarity $\perp$ invariant (i.e., $h \in H$ if and only if ${(R^h)}^\perp = {(R^\perp)}^h$). We shall find it helpful to work with the elements of $\PGL(3,q)$ as matrices in $\GL(3,q)$ and the points of $\PG(2,q)$ as column vectors, with matrices acting on the left. 

\begin{lemma}\cite[Lemma 8.3.6]{H}
$H \cong \PGL(2,q)$ and the isomorphism is given by
$$
\left(
\begin{array}{cc}
a & b \\
c & d \\
\end{array}
\right)
\longleftrightarrow
\left(
\begin{array}{cccc}
1 & 0 & 0 \\
0 & a & b \\
0 & c & d \\
\end{array}
\right),
$$
with $a, b, c, d \in \Fq$, $ad + bc = 1$.
\end{lemma}
We introduce the following definition. In $\PG(2,q)$, a {\em maximal arc} $\cA$ of {\em degree} $n$ is a subset consisting of $(n-1)q+n$ points of the plane such that every line meets $\cA$ in $0$ or $n$ points, for some $n$. 

In the following, $Tr$ will denote the usual {\em absolute trace function} from $\Fq$ to $\mathbb{F}_2$. Let $\alpha \in \Fq$ such that $Tr(\alpha) = 1$. Then the polynomial $X^2 + X +\alpha = 0$ is irreducible over $\Fq$. Let $\lambda \in \Fq$ and consider the conic given by $\cC_\lambda : X_2^2 + X_2 X_3 + \alpha X_3^2 + \lambda X_1^2 = 0$. Then the set $\{\cC_\lambda \,\, | \,\, \lambda \in \Fq \} \cup \{\ell\}$ forms a pencil $\cF$ giving rise to a partition of the points of the plane. Every conic $\cC_\lambda$, $\lambda \neq 0$, in the pencil $\cF$ has as nucleus the point $\cC_0 = U_1$. The pencil $\cF$ is stabilized by the following cyclic group of order $q+1$; the orbits being the conics of the pencil,
$$
C = \left\{ 
\left(
\begin{array}{cccc}
1 & 0 & 0 \\
0 & a & \alpha b \\
0 & b & a+b \\
\end{array}
\right) \;\; | \;\; a^2+ab+\alpha b^2 = 1 \right\} .
$$
In \cite{D}, R. H. F. Denniston proved that if $A$ is an additive subgroup of $\Fq$ of order $n$, then the set of points of all $\cC_\lambda$, for $\lambda \in A$, form a maximal arc of degree $n$ in $\PG(2,q)$. On the other hand, from \cite[Theorem 2.2]{AL}, if a maximal arc in $\PG(2,q)$ is invariant under a linear collineation group of $\PG(2,q)$ which is cyclic and has order $q+1$, then it is a Denniston maximal arc.
In the case when $q$ is even, the current best bound for the independence number of $ER_q$ is in \cite{HW, MW}. If $q$ is an even power of $2$, then 
\begin{equation} \label{b}
q^{3/2}-q+\sqrt{q} \le \alpha(ER_q) \le q^{3/2}-q+\sqrt{q} + 1 .
\end{equation}
The upper bound in \eqref{b} comes from \cite{HW}, while the lower bound comes from $\cite{MW}$. In \cite{MW}, the authors show the existence of a maximal arc $\cA$ of Denniston type of degree $\sqrt{q}$ of $\PG(2,q)$, such that the points of $\cA$ correspond to a coclique of $ER_q$.
On the other hand, from \cite{MW}, if $q$ is an odd power of $2$, then 
\begin{equation} 
\alpha(ER_q) \ge \frac{q^{3/2}}{2 \sqrt{2}} .
\end{equation}
Here we give a proof for an improved lower bound:
$$
\alpha(ER_q) \ge \frac{q^{3/2}}{\sqrt{2}}-q+\sqrt{\frac{q}{2}}.
$$
In this case we will show the existence of a maximal arc $\cA$ of Denniston type of degree $\sqrt{q/2}$ of $\PG(2,q)$, such that the points of $\cA$ correspond to a coclique of $ER_q$.
First, we prove the following lemma.
\begin{lemma}\label{conics}
If $\lambda \in \Fq$, with $Tr(\lambda) = 0$, then for every point $R \in \cC_{\lambda^2}$, $| R^\perp \cap \cC_{\lambda^2} | = 0$.
\end{lemma}
\begin{proof}
With the notation introduced above, the point $P_\lambda = (1, \lambda, 0) \in \cC_{\lambda^{2}} = \{(1, \lambda a, \lambda b) \;\; | \;\; a^2 + ab + \alpha b^2 = 1 \}$. Hence, $P_\lambda^C = \cC_{\lambda^2}$. The line $P_\lambda^\perp$ has equation $X_1 + \lambda X_3 = 0$. We consider the intersection $P_\lambda^\perp \cap \cC_{\lambda^2}$. If $\lambda = 0$, the assertion is trivial. The point $(1, \lambda a, \lambda b) \in \cC_{\lambda^2}$, $\lambda \neq 0$, belongs to $P_{\lambda}^\perp$ if and only if $b = 1 / \lambda^2$ and $a=x / \lambda^2$, where $x$ is a solution of 
\begin{equation} \label{ee}
X^2 + X + (\alpha + \lambda^4) = 0.
\end{equation}
The equation in \eqref{ee} has two or zero solutions according as $Tr(\alpha + \lambda^4) = 0$ or $1$ respectively. On the other hand, $Tr(\alpha + \lambda^4) = 1$ if and only if $Tr(\lambda^4) = 0$ if and only if $Tr(\lambda) = 0$.
\end{proof}
\begin{theorem}
If $q$ is an odd power of $2$, then there exists a maximal arc $\cA$ of Denniston type of degree $\sqrt{q/2}$ of $\PG(2,q)$, such that for every point $P \in \cA$, $| P^\perp \cap \cA | = 0$.
\end{theorem}
\begin{proof}
From Lemma \ref{conics}, there are $q/2-1$ possibilities for $\lambda \neq 0$ such that, for every point $R \in P_{\lambda}^C$, we have that $| R^\perp \cap P_{\lambda}^C | = 0$. Let $\lambda_1, \lambda_2$ be two distinct non--zero elements of $\Fq$ such that $Tr(\lambda_1) = Tr(\lambda_2) = 0$. We consider the intersection $P_{\lambda_1}^\perp \cap \cC_{\lambda_2^2}$. Again, the point $(1, \lambda_2 a, \lambda_2 b) \in \cC_{\lambda_2^2}$ belongs to $P_{\lambda_1}^\perp$ if and only if $b = 1/\lambda_1 \lambda_2$ and $a=x / \lambda_1 \lambda_2$, where $x$ is a solution of 
\begin{equation} \label{ee1}
X^2 + X + (\alpha + \lambda_1^2 \lambda_2^2) = 0 .
\end{equation}
The equation in \eqref{ee1} has two or zero solutions according as $Tr(\alpha + \lambda_1^2 \lambda_2^2) = 0$ or $1$ respectively. On the other hand, $Tr(\alpha + \lambda_1^2 \lambda_2^2) = 1$ if and only if $Tr(\lambda_1^2 \lambda_2^2) = 0$ if and only if $Tr(\lambda_1 \lambda_2) = 0$. Let $q = 2^n$, $n$ odd. To conclude the proof we need to show the existence of a subset $N \subset \Fq$, with $| N | = 2^{(n-1)/2}$, such that for every $x_1, x_2 \in N$, we have that $Tr(x_1 x_2) = 0$. In order to do that, we note that it is possible to identify the non--zero elements of $\mathbb{F}_{2^n}$ with $\PG(n-1, 2)$ and, in this setting, the absolute trace function
$$
(x_1, x_2) \in \mathbb{F}_2^n \times \mathbb{F}_2^n \longmapsto Tr(x_1 x_2) \in \mathbb{F}_2
$$  
defines a non--degenerate symmetric bilinear form $f$ on $\mathbb{F}_2^n$. Since $n$ is odd, the form $f$ gives rise to a pseudo polarity $p$ of $\PG(n-1,2)$. This means that we have a hyperplane $H$ of absolute points, i.e. points $x \in \PG(n-1,2)$ such that $Tr(x)=0$. When restricting the polarity $p$ to this hyperplane $H$, we obtain a symplectic polarity $p|_H$. Let $M$ be a maximal subspace of $H$ consisting of absolute points with respect to $p|_H$ and let $N'$ be the set of elements of $\mathbb{F}_2^n$ corresponding to points of $M$. It is known that $M$ is a projective subspace of dimension $(n-3)/2$, from which follows that $N \coloneqq N' \cup \{ 0 \}$ is a subset of $\GF(2^n)$ of size $2^{(n-1)/2}$ having the required property. 

Lastly, we note that $A = \left\{\lambda^2 \,\, | \,\, \lambda \in N \right\}$ is an additive subgroup of $\Fq$. Therefore, the set of points of all $\cC_\lambda$, with $\lambda \in A$, is indeed a maximal arc of Denniston type of degree $|A| = \sqrt{q/2}$.
\end{proof}

\begin{cor}\label{coreven}
If $q$ is an odd power of $2$, then 
$$
\alpha(ER_q) \ge \frac{q^{3/2}}{\sqrt{2}}-q+\sqrt{\frac{q}{2}}
$$
\end{cor}

\begin{remark}
The number of lines disjoint from the maximal arc $\cA$ equals $\sqrt{2q}(q-\sqrt{q/2}+1)$, which is more than the number of disjoint lines of type $P^\perp$, for some $P \in \cA$. Therefore, there are $\sqrt{q/2}(q+1)$ points such that their polar line is disjoint from $\cA$ and so each of them can be added to enlarge the coclique. 
\end{remark}

\section{On the independence number of the Erd\H os--R\'enyi hypergraph of triangles}

A \emph{hypergraph} $\Gamma$ is a family of distinct subsets of a finite set. The members of $\Gamma$ are called {\em edges}, and the elements of $V(\Gamma) = \bigcup_{E \in \Gamma} E$ are called {\em vertices}. If all edges in $\Gamma$ have size $r$, then $\Gamma$ is called an $r$--{\em uniform hypergraph} or, simply, $r$--{\em graph}. For example, a $2$--graph is a graph in the usual sense. A vertex $v$ and an edge $E$ are called \emph{incident} if $v \in E$. The \emph{degree} of a vertex $v$ of $\Gamma$, denoted $d(v)$, is the number of edges of $\Gamma$ incident with $v$. 

For $k \geq 2$, a \emph{cycle} of length $k$ in a hypergraph $\Gamma$ is an alternating sequence of vertices and edges of the form $v_1, E_1, v_2, E_2, \dots, v_k, E_k, v_1$ such that
\begin{enumerate}
	\item $v_1, v_2, \dots, v_k$ are distinct vertices of $\Gamma$
	\item $E_1, E_2, \dots, E_k$ are distinct edges of $\Gamma$
	\item $v_i,v_{i+1} \in E_i$ for each $i \in \left\{1,2,\dots,k-1\right\}$ and $v_k,v_1 \in E_k$.
\end{enumerate}
The {\em girth} of a hypergraph $\Gamma$, containing a cycle, is the minimum length of a cycle in $\Gamma$. 

The \emph{generalized Tur\'an number} $T_r(n, k, l)$ is defined to be the maximum number of edges in an $r$--graph on $n$ vertices in which no set of $k$ vertices spans $l$ or more edges. The asymptotic behaviour of the numbers $T_r(n, k, l)$, in general, is unknown, and seems to be difficult to determine, see \cite{C}. The value $T_3(n, 8, 4)$ gives the maximum number of edges in a $3$--graph of girth five. This is seen by directly checking that any four triples on a set of eight vertices span a hypergraph containing a cycle of length at most four. 

In \cite{LV}, the authors studied a hypergraph $\cH_q$ of girth $5$ constructed from $ER_q$. The hypergraph $\cH_q$ is the $3$--graph whose vertex set is the set of non--absolute points of $V(ER_q)$ and edge set is the set of triangles in $ER_q$. These hypergraphs were used to determine the asymptotics of the Tur\'an number $T_3(n,8,4)$. It follows that $\alpha(\cH_q)$ is the order of the largest triangle--free induced subgraph of $ER_q$ which contains no absolute points. 

In \cite{P}, Parsons constructs a triangle--free induced subgraph of $ER_q$, $q$ odd, which contains no absolute points and has either $q(q+1)/2$ or $q(q-1)/2$ vertices according as $q \equiv -1\pmod 4$ or $q \equiv 1\pmod 4$, respectively. See also \cite[Theorem 8.3.4]{H}. From \cite[Theorem 8.3.5]{H}, the hypergraph $\cH_q$ contains $q(q^2-1)/6$ edges. The next result shows the existence of a triangle--free induced subgraph $\cS$ of $ER_q$, $q$ even, which contains no absolute points and has $q(q+1)/2$ vertices. As a consequence we establish the asymptotic tightness of the following bound determined in \cite[Theorem 8]{MW}: 
$$
\alpha(\cH_q) \le \frac{q^2}{2} + q^{3/2} + O(q) .
$$
Let $q$ be even and let $\perp$ denote the pseudo polarity of $\PG(2,q)$ described in Section $\ref{even}$. In order to construct $\cS$, with the notation introduced in Section \ref{even}, we consider the orbits of the cyclic group $C$ of order $q+1$ on lines of $\PG(2,q)$. Since the $C$--orbit of a point $R$ not belonging to $\ell \cup \{ U_1 \}$ is a conic and $C$ leaves the polarity $\perp$ invariant, we have that the $C$--orbit of a line $r$ not containing $\{ U_1 \}$ and distinct from $\ell$ is a dual conic. Each such a dual conic $\cD$ left invariant by the group $C$, has as dual nucleus the line $\ell$. This means that $\cD$ consists of $q+1$ lines such that every point of $\PG(2,q) \setminus \ell$ is contained in either $0$ or $2$ lines of $\cD$ and every point of $\ell$ is contained in exactly one line of $\cD$. From Lemma \ref{conics}, there are $q/2-1$ possibilities for $\lambda \neq 0$ such that, for every point $R \in P_{\lambda}^C$, we have that $| R^\perp \cap P_{\lambda}^C | = 0$. We want to prove that if $Tr(\lambda) = 0$, with $\lambda \neq 0$, then the $q(q+1)/2$ points of $\PG(2,q) \setminus \ell$ covered by $({P_\lambda^\perp})^C$ can be chosen as vertices of $\cS$.
  
\begin{theorem} \label{triangle}
If $q$ is even, there exists a triangle--free induced subgraph of $ER_q$, $q$ even, which contains no absolute points and has $q(q+1)/2$ vertices.
\end{theorem}
\begin{proof}
With the notation introduced in Section \ref{even}, let $P_\lambda = (1,\lambda,0)$, $\lambda \in \Fq \setminus \{ 0 \}$, with 
\begin{equation}\label{tr0}
Tr(\lambda) = 0
\end{equation}
then $P_\lambda^C = \cC_{\lambda^2}$, where $\cC_{\lambda^2}$ is the conic having equation $X_2^2 + X_2 X_3 + \alpha X_3^2 + \lambda^2 X_1^2 = 0$. Let $\cD$ be the dual conic ${P_{\lambda}^\perp}^C$ and let $\cS$ be the set consisting of the $q(q+1)/2$ points of $\PG(2,q) \setminus \ell$ covered by ${P_\lambda^\perp}^C$. Note that a point of $\PG(2,q) \setminus \ell$ lies on two lines of $\cD$ if and only if its polar line under the pseudo polarity is secant to $\cC_{\lambda^2}$. Let $R = (1, y, x) \in \PG(2,q) \setminus (\ell \cup \{ U_1 \})$, then $R^\perp : X_1 + x X_2 + y X_3 = 0$. The line $R^\perp$ contains $2$ points of $\cC_{\lambda^2}$ if and only if the equation
$$
(1+\lambda^2 x^2) X_2^2 + X_2 X_3 + (\alpha + \lambda^2 y^2) X_3^2 = 0
$$  
has two solutions in $\Fq$, i.e., if and only if 
$$
Tr(\alpha + \alpha \lambda^2 x^2 + \lambda^2 y^2 + \lambda^4 x^2 y^2) = 0
$$
if and only if 
\begin{equation}\label{tr}
Tr(\alpha \lambda^2 x^2 + \lambda^2 y^2 + \lambda^4 x^2 y^2) = 1 .
\end{equation}
Let $x,y \in \Fq$ such that \eqref{tr} is satisfied. Then the point $R$ lies on two lines of $\cD$. 

\medskip
\fbox{Case 1): $y \neq 0$}
\medskip

\noindent
Let $U = (1, \mu, (\mu x + 1)/y)$, with $\mu \in \Fq$, be a generic point of $R^\perp \setminus \ell$ and let $U^\perp : X_1 +  \frac{\mu x + 1}{y} X_2 + \mu X_3 = 0$. Then, again, the line $U^\perp$ contains $2$ points of $\cC_{\lambda^2}$ if and only if the equation
$$
\left(1+\lambda^2  \frac{(\mu x + 1)^2}{y^2}\right) X_2^2 + X_2 X_3 + (\alpha + \lambda^2 \mu^2) X_3^2 = 0
$$  
has two solutions in $\Fq$, i.e., if and only if 
\begin{equation}\label{tr1}
Tr\left(\left(1 + \lambda^2  \frac{(\mu x + 1)^2}{y^2}\right)\left(\alpha + \mu^2 \lambda^2\right)\right) = 0 .
\end{equation}
Since $R \not\in \cC_{\lambda^2}$, it follows that $R^\perp$ is not a line of $\cD$. Hence, $R^\perp$ contains $q/2$ points of $\cS$. It turns out that $\mu$ can be chosen in $q/2$ ways such that \eqref{tr1} is satisfied. A straightforward calculation shows that the lines $R^\perp$ and $U^\perp$ intersect in the point $V = (1, \mu + y, (1+\mu x + x y)/y)$. The unique triangle of $ER_q$ containing the points $R$ and $U$ is the triangle having as vertices the points $R, U, V$. Moreover, the point $V$ belongs to $\cS$ if and only if the equation
$$
\left(1+\lambda^2  \frac{(\mu x + 1 + x y)^2}{y^2}\right) X_2^2 + X_2 X_3 + \left(\alpha +  (\mu + y)^2 \lambda^2\right) X_3^2 = 0
$$  
has two solutions in $\Fq$, i.e., if and only if 
\begin{equation}\label{tr2}
Tr\left(\left(1 + \lambda^2  \frac{(\mu x + 1 + x y)^2}{y^2}\right)\left(\alpha + (\mu + y)^2 \lambda^2\right)\right) = 0 .
\end{equation}  
On the other hand, since
$$
\left(1 + \lambda^2  \frac{(\mu x + 1 + x y)^2}{y^2}\right) \left( \alpha + (\mu + y)^2 \lambda^2 \right) = 
$$
$$
\left(1 + \lambda^2  \frac{(\mu x + 1)^2}{y^2}\right) \left( \alpha + \mu^2 \lambda^2 \right) + (\alpha \lambda^2 x^2 + \lambda^2 y^2 + \lambda^4 x^2 y^2) + \lambda^4 , 
$$
taking into account \eqref{tr0}, \eqref{tr}, \eqref{tr1}, it is easily seen that \eqref{tr2} is never satisfied. 

\medskip
\fbox{Case 2): $y = 0$}
\medskip

\noindent
Let $U' = (x, 1, \mu)$, with $\mu \in \Fq$, be a generic point of $R^\perp \setminus \ell$ and let $U'^\perp : x X_1 +  \mu X_2 + X_3 = 0$. Then, again, the line $U'^\perp$ contains $2$ points of $\cC_{\lambda^2}$ if and only if the equation
$$
\left(1+  \frac{\lambda^2 \mu^2}{x^2}\right) X_2^2 + X_2 X_3 + \left(\alpha + \frac{\lambda^2}{x^2}\right) X_3^2 = 0
$$  
has two solutions in $\Fq$, i.e., if and only if 
\begin{equation}\label{tr3}
Tr\left(\left(1+  \frac{\lambda^2 \mu^2}{x^2}\right)\left(\alpha + \frac{\lambda^2}{x^2}\right)\right) = 0 .
\end{equation}
Since $R \not\in \cC_{\lambda^2}$, it follows that $R^\perp$ is not a line of $\cD$. Hence, $R^\perp$ contains $q/2$ points of $\cS$. It turns out that $\mu$ can be chosen in $q/2$ ways such that \eqref{tr3} is satisfied. A straightforward calculation shows that the lines $R^\perp$ and $U'^\perp$ intersect in the point $V' = (x, 1, x^2+\mu)$. The unique triangle of $ER_q$ containing the points $R$ and $U'$ is the triangle having as vertices the points $R, U', V'$. Moreover, the point $V'$ belongs to $\cS$ if and only if the equation
$$
\left(1+\lambda^2 x^2 + \frac{\lambda^2 \mu^2}{x^2}\right) X_2^2 + X_2 X_3 + \left(\alpha + \frac{\lambda^2}{x^2}\right) X_3^2 = 0
$$  
has two solutions in $\Fq$, i.e., if and only if 
\begin{equation}\label{tr4}
Tr\left(\left(1+\lambda^2 x^2 + \frac{\lambda^2 \mu^2}{x^2}\right)\left(\alpha + \frac{\lambda^2}{x^2}\right)\right) = 0 .
\end{equation}  
On the other hand, since
$$
\left(1+\lambda^2 x^2 + \frac{\lambda^2 \mu^2}{x^2}\right)\left(\alpha + \frac{\lambda^2}{x^2} \right) = 
\left(1+ \frac{\lambda^2 \mu^2}{x^2}\right)\left(\alpha + \frac{\lambda^2}{x^2} \right) + \left(\alpha \lambda^2 x^2\right) + \lambda^4 , 
$$
taking into account \eqref{tr0}, \eqref{tr}, \eqref{tr3}, it is easily seen that \eqref{tr4} is never satisfied. 
\end{proof}
\begin{remark}
By construction the subgraph $\cS$ admits the cyclic group $C$ of order $q+1$ as a group of automorphisms.
\end{remark}
\begin{remark}
Theorem \ref{triangle} provides a solution to \cite[Open Problem 3]{MW}.
\end{remark}

\begin{cor}
If $q$ is even, there exists a $\frac{q}{2}$--regular graph on $q(q+1)/2$ vertices of girth at least $5$.
\end{cor}
\begin{proof}
Taking into account Theorem \ref{triangle} and the fact that $ER_q$ does not contain $C_4$, it is enough to show that the graph $\cS$ constructed above is regular. Let $R$ be a point of $\PG(2,q)$ corresponding to a vertex of $\cS$, then a point $R'$ distinct from $R$ is adjacent with $R$ if and only if $R' \in (R^\perp \cap r) \setminus \ell$, where $r$ is a line of $\cD$. Since $\cD$ contains $q+1$ lines and through $\ell \cap R^\perp$ there passes exactly one line of $\cD$, we have that there are exactly $q/2$ points in $R^\perp \setminus \ell$ such that through each of them there pass two lines of $\cD$.      
\end{proof}


\end{document}